\documentclass[12pt]{amsart}

\usepackage{amsmath}
\usepackage{amsfonts}
\usepackage{amssymb}
\usepackage{amsthm} 
\usepackage[english]{babel}
\usepackage[latin1]{inputenc}
\usepackage[T1]{fontenc}
\usepackage{graphicx}
\usepackage{enumitem}

\usepackage{geometry}
\usepackage{hyperref}

\usepackage{xcolor}

\DeclareMathOperator{\Ric}{Ric}

\DeclareMathOperator{\spt}{spt}

\DeclareMathOperator{\Vol}{Vol}

\DeclareMathOperator{\Tr}{Tr}
\DeclareMathOperator{\Conv}{Conv}

\DeclareMathOperator{\Hess}{Hess}

\def\Z{\mathbb Z}

\newtheorem{theo}{Theorem}[]

\newtheorem{lemme}[theo]{Lemma}

\begin{document}

\title{Equidistribution of minimal hypersurfaces for generic metrics} 
\author{Fernando C. Marques, Andr\'e Neves and Antoine Song}
\address{Princeton University \\ Fine Hall \\ Princeton NJ 08544 \\ USA}
\email{coda@math.princeton.edu}
\address{University of Chicago \\ Department of Mathematics \\ Chicago IL 60637\\ USA /Imperial College London\\ Huxley Building \\ 180 Queen's Gate \\ London SW7 2RH \\ United Kingdom}
\email{aneves@uchicago.edu, a.neves@imperial.ac.uk}
\address{Princeton University \\ Fine Hall \\ Princeton NJ 08544 \\ USA}
\email{aysong@math.princeton.edu}
\thanks{The first author is partly supported by NSF-DMS-1509027 and NSF DMS-1311795. The second author is partly supported by NSF  DMS-1710846 and EPSRC Programme Grant EP/K00865X/1. The third author is supported by NSF-DMS-1509027.}

\maketitle

\begin{abstract}
For almost all Riemannian metrics (in the $C^\infty$ Baire sense) on a closed manifold $M^{n+1}$, $3\leq (n+1)\leq 7$, we prove that there is a sequence of  closed, smooth, embedded, connected minimal hypersurfaces that is equidistributed in $M$.  

This gives a quantitative version of the main result of  \cite{irie-marques-neves}, by Irie and the first two authors, that established density of minimal hypersurfaces for generic metrics. As in \cite{irie-marques-neves}, the main tool is the Weyl Law for the Volume Spectrum proven by Liokumovich and the first two authors in \cite{liokumovich-marques-neves}.
 
 \end{abstract}
 
 \section{Introduction}
%correction1:replaced 2-tensor by (0,2)-tensor

In 1982, S. T. Yau (\cite{yau1}) conjectured that every closed Riemannian three-manifold contains infinitely many smooth, closed, immersed minimal surfaces. In \cite{irie-marques-neves}, Irie and the first two authors settled Yau's conjecture in the generic case by proving a much stronger property holds true:

\medskip

{\bf Theorem} (Irie, Marques, Neves, 2017): \textit{Let $M^{n+1}$ be a closed manifold of dimension $(n+1)$, with $3\leq (n+1)\leq 7$. Then for a $C^\infty$-generic Riemannian metric $g$ on $M$, the union of all closed, smooth, embedded minimal hypersurfaces is dense.}
\medskip

In our paper, we use the methods of \cite{irie-marques-neves} in a more quantitative way and prove an even stronger property: there is a sequence of closed, smooth, embedded, connected minimal hypersurfaces that is equidistributed in $M$.

\medskip

%correction2: changed statement
{\bf Main Theorem}: \textit{Let $M^{n+1}$ be a closed manifold of dimension $n+1$, with $3\leq (n+1)\leq 7$. Then for a $C^\infty$-generic Riemannian metric $g$ on $M$, there exists a sequence $\{\Sigma_j\}_{j\in \mathbb{N}}$ of closed, smooth, embedded, connected minimal hypersurfaces  that is equidistributed in $M$:  for any $f \in C^\infty(M)$ we have
\begin{equation}\label{equidistribution.function}
\lim_{q \rightarrow \infty} \frac{1}{\sum_{j=1}^{q} {\rm vol}_g(\Sigma_{j})}\sum_{j=1}^{q} \int_{\Sigma_{j}} f \, d\Sigma_{j}=\frac{1}{{\rm vol}_g M}\int_M f dM.
\end{equation}
Even more, for any symmetric $(0,2)$-tensor $h$ on $M$ we have:
\begin{equation}\label{equidistribution.tensor}
\lim_{q \rightarrow \infty} \frac{1}{\sum_{j=1}^{q} {\rm vol}_g(\Sigma_{j})}\sum_{j=1}^{q} \int_{\Sigma_{j}} \Tr_{\Sigma_{j}} (h) \, d\Sigma_{j}=\frac{1}{{\rm vol}_g M}\int_M \frac{n\Tr_M h}{n+1} dM.
\end{equation}
}
\medskip

Equidistribution theorems have an old history in fields like number theory, ergodic theory and harmonic analysis. Equidistribution of closed geodesics is known in some cases, like for compact hyperbolic manifolds (Bowen' 72 \cite{bowen} or Margulis \cite{margulis}, see also \cite{zelditch}). Equidistribution  of totally geodesic surfaces is a well-studied problem for hyperbolic $3$-manifolds (\cite{eskin-hee-oh},  \cite{mcmullen-mohammadi-oh}, \cite{mozes-shah}, \cite{ratner}, \cite{shah}). Our theorem is the first of its kind for the higher-dimensional setting of minimal surfaces in general manifolds.

{\bf Remark:} Yau's Conjecture has been fully resolved by the third author \cite{song-infinitely-many}. He  was able to localize the methods initially developed by the first two authors in \cite{marques-neves-infinitely}, and proved that any compact $(M^{n+1},g)$, $3\leq (n+1) \leq 7$, contains infinitely many smooth, embedded, closed minimal hypersurfaces. It would be interesting to know whether density and equidistribution of minimal hypersurfaces hold for all Riemannian metrics. 

\medskip

As in the Irie-Marques-Neves paper (\cite{irie-marques-neves}), the crucial tool in our proof is the Weyl law for the Volume Spectrum conjectured by Gromov (\cite{gromov}) and recently proven by the first two authors jointly with Liokumovich in \cite{liokumovich-marques-neves}:
\medskip

{\bf Weyl Law for the Volume Spectrum} (Liokumovich, Marques, Neves, 2016): {\it There exists a universal constant $a(n)>0$ such that for any compact Riemannian manifold $(M^{n+1},g)$ we have:
%correction1: replaced k by p
$$
\lim_{p\rightarrow \infty} \omega_p(M,g)p^{-\frac{1}{n+1}} = a(n) {\rm vol}(M,g)^\frac{n}{n+1}.
$$
}

The volume spectrum of a compact Riemannian manifold $(M^{n+1},g)$ is a nondecreasing sequence of numbers $\{\omega_p(M,g):p\in \mathbb{N}\}$ defined variationally by performing a min-max procedure for the area functional over  multiparameter sweepouts. The first estimates for these numbers were proven in fundamental papers by Gromov in the late 1980s \cite{gromov0} and by Guth \cite{guth} more recently.

Our proof also uses a transversality argument, based on the Structure Theorem of White (\cite{white2}, Theorem 2.1), that allows one to compute the derivative of the $p$-width as the derivative of the area of some minimal hypersurface.
We combine this information with appropriately chosen $N$-parameter deformations of the metric, for $N$ large, that generalize the one-parameter deformations of \cite{irie} and \cite{irie-marques-neves}. A key idea in the paper (that can be be deduced from Lemma \eqref{differential}) is that metrics which  are critical points of the functional $g\mapsto \omega_p(M,g)p^{-\frac{1}{n+1}}-a(n) {\rm vol}(M,g)^\frac{n}{n+1}$ when restricted to the $N$-parameter family of  deformations have  minimal hypersurfaces that obey some form of equidistribution. The fact that this functional is only Lipschitz continuous and thus not differentiable everywhere is a serious technical issue that the authors had to overcome.

We note that Property (\ref{equidistribution.function}) is equivalent to saying that 
$$
\frac{\sum_{j=1}^q \mu_{\Sigma_j}}{\sum_{j=1}^q \mu_{\Sigma_j}(M)} \rightarrow \frac{\mu}{\mu(M)}
$$
as measures, where $\mu_{\Sigma_j}=||\Sigma_j||$ is the Radon measure $\mu_{\Sigma_j}(U)={\rm vol}_g(\Sigma_j \cap U)$,  $U \subset M$, and $\mu=dv_g$ is the Riemannian volume measure of $(M,g)$.  Property (\ref{equidistribution.function}) follows from Property (\ref{equidistribution.tensor}) by choosing $h=f \cdot g$.

The dimensional restriction in the Main Theorem  is due to the fact that in higher dimensions min-max (even area-minimizing) minimal hypersurfaces can have singular sets. We use Almgren-Pitts theory (\cite{almgren-varifolds}, \cite{pitts}), which together with Schoen-Simon regularity (\cite{schoen-simon}) produces smooth minimal hypersurfaces when $3\leq (n+1) \leq 7$. We expect that the methods of this paper can be generalized to handle the higher-dimensional singular case.

The Main Theorem raises many interesting and exciting new questions: namely whether equidistribution holds in the Grassmanian bundle, whether the minimal hypersurfaces realizing the width are the ones that become equidistributed, or whether there are conditions  which ensure that the sequence of unit measures $\mu_{\Sigma_j}/\mu_{\Sigma_j}(M)$ converges to the normalized volume measure. Any progress related with these directions would be highly desirable.

\section{Preliminaries}

We suppose that $M$ is a closed manifold of dimension $3 \leq (n+1) \leq 7$. For each $2\leq q\leq \infty$, we denote by $\Gamma_q$ the space of all $C^q$ Riemannian metrics on $M$, endowed with the $C^q$ topology. Given $g \in \Gamma_q$, we let $\mathcal{V}(g)$ be the set of stationary integral varifolds in $(M,g)$ whose support is a closed, $C^2$,  embedded, minimal hypersurface. Hence $V \in \mathcal{V}(g)$ if and only if there exist a disjoint collection $\{\Sigma_1,\dots,\Sigma_s\}$ of closed, $C^2$,  embedded, connected minimal hypersurfaces in $(M,g)$ and integers $\{m_1, \dots, m_s\}\subset \mathbb{N}$ such that $V = m_1 \Sigma_1 + \cdots + m_s \Sigma_s$. By elliptic regularity, each $\Sigma_i$ is in fact of class $C^q$. The support of $V$ is denoted by $\spt(V)$ and is equal to $\cup_{i=1}^s \Sigma_i$, while $||V||$ denotes the Radon measure induced by $V$ on $M$.

We denote by $\mathcal{Z}_n(M;\Z_2)$ the space of modulo two $n$-dimensional flat chains $T$ in $M$  with $T=\partial U$ for some $(n+1)$-dimensional modulo two flat chain $U$ in $M$, endowed with the flat topology. This space is weakly homotopically equivalent to $\mathbb{RP}^\infty$ (see Section 4 of \cite{marques-neves-topology}). We denote by $\overline{\lambda}$ the generator of $H^1(\mathcal{Z}_n(M;\Z_2), \mathbb{Z}_2)=\mathbb{Z}_2$. The mass ($n$-dimensional volume) of $T$ is denoted by $M(T)$.

Let $X$ be a finite dimensional simplicial complex. A continuous map $\Phi:X\rightarrow \mathcal{Z}_n(M;\Z_2)$ is called a {\em  $p$-sweepout} if
$$
\Phi^*(\bar \lambda^p) \neq 0 \in H^p(X;\Z_2).
$$
We say $X$ is {\em $p$-admissible} if there exists  a $p$-sweepout $\Phi:X\rightarrow \mathcal{Z}_n(M;\Z_2)$  that has no concentration of mass, meaning 
$$\lim_{r\to 0} \sup\{M(\Phi(x) \cap B_r(p)):x\in X, p\in M\}=0.$$
The set of all $p$-sweepouts $\Phi$ that have no concentration of mass is denoted by $\mathcal P_p$. Note that two maps in  $\mathcal P_p$ can have different domains.

In \cite{marques-neves-infinitely}, the first two authors  defined
\medskip

{\bf Definition:} The {\em $p$-width of $(M,g)$}  is the number
$$\omega_p(M,g)=\inf_{\Phi \in \mathcal P_p}\sup\{M(\Phi(x)): x\in {\rm dmn}(\Phi)\},$$
where ${\rm dmn}(\Phi)$ is the domain of $\Phi$.

The next lemma gives that the normalized $p$-width $p^{-\frac{1}{(n+1)}}\omega_p(M,g)$ is a  Lipschitz function of the metric on sets of uniformly equivalent metrics, with a Lipschitz constant that does not depend on $p$.
\begin{lemme} \label{lipschitz}
Let  $\tilde{g}$ be a $C^2$ Riemannian metric on $M$, and let  $C_1<C_2$ be positive constants. Then there exists $K=K(\tilde{g},C_1,C_2)>0$ such that
$$
|p^{-\frac{1}{(n+1)}}\omega_p(M,g)-p^{-\frac{1}{(n+1)}}\omega_p(M,g')| \leq K \cdot |g-g'|_{\tilde{g}}
$$
for any $g,g'\in \{h\in \Gamma_2 ; C_1\tilde{g} \leq h \leq C_2 \tilde{g}\}$ and any $p \in \mathbb{N}$. %$p \in \mathbb{N}\backslash\{0\}$.

\end{lemme}

\begin{proof}
It follows from the Gromov-Guth bound (\cite{gromov0}, \cite{guth}, see Theorem 5.1 of  \cite{marques-neves-infinitely})  that there exists $C=C(\tilde{g})$ such that $\omega_p(M, \tilde{g}) \leq Cp^{\frac{1}{(n+1)}}$ for every $p\in \mathbb{N}$. 

Given $g,g'\in \{h\in \Gamma_2 ; C_1\tilde{g} \leq h \leq C_2 \tilde{g}\}$, one can check (see Lemma 2.1 of \cite{irie-marques-neves}) that
\begin{eqnarray*}
&&\omega_p(M,g')-\omega_p(M,g) \leq \left(\left(\sup_{v \neq 0} \frac{g'(v,v)}{g(v,v)}\right)^\frac{n}{2}-1\right)\omega_p(M,g)\\
&& \leq \left(\left(1 + \sup_{v \neq 0} \frac{|g(v,v)-g'(v,v)|}{g(v,v)}\right)^\frac{n}{2}-1\right) \omega_p(M,g)\\
&& \leq \left(\left(1 +C_1^{-1}  |g-g'|_{\tilde{g}}\right)^\frac{n}{2}-1\right) \omega_p(M,g)\\
&& \leq \left(\left(1 +C_1^{-1}  |g-g'|_{\tilde{g}}\right)^\frac{n}{2}-1\right) C_2^\frac{n}{2}\omega_p(M,\tilde{g})\\
&& \leq \left(\left(1 +C_1^{-1}  |g-g'|_{\tilde{g}}\right)^\frac{n}{2}-1\right) C_2^\frac{n}{2}Cp^\frac{1}{n+1},\\
\end{eqnarray*}
from which the result follows.

\end{proof}

The next lemma concerns the differentiability properties of the $p$-width restricted to a generic finite-dimensional family of metrics. Let $I^N=[0,1]^N$.
\begin{lemme} \label{differential}
Let $g:I^N \rightarrow \Gamma_q$ be a smooth embedding, $N \in \mathbb{N}$. If $q \geq N+3$, then there exists an arbitrarily small perturbation in the $C^\infty$ topology  $g':I^N \rightarrow \Gamma_q$ of $g$ such that  there is a subset $\mathcal{A}\subset I^N$ of full $N$-dimensional Lebesgue measure with the following property: for any $p\in \mathbb{N}$ and  any point $t$ of $\mathcal{A}$, the function $s \mapsto \omega_p(g'(s))$ is differentiable at $t$ and there exists a disjoint collection $\{\Sigma_1, \dots, \Sigma_Q\}$ of closed, $C^q$, embedded, minimal hypersurfaces of $(M,g'(t))$ together with integers $\{m_1, \dots, m_Q\} \subset \mathbb{N}$ so that
$$
\omega_p(g'(t)) = \sum_{j=1}^Q m_j {\rm vol}_{g'(t)}(\Sigma_j), \, \, \,  \sum_{j=1}^Q {\rm index}(\Sigma_j) \leq p,
$$
and
\begin{eqnarray*}
\frac{\partial}{\partial v}(\omega_p \circ g')_{|s=t}&=&\frac{\partial}{\partial v} \left(\sum_{j=1}^Q m_j {\rm vol}_{g'(s)}(\Sigma_j)\right)_{|s=t}\\
&=& \sum_{j=1}^Q m_j \int_{\Sigma_j} \frac12 \Tr_{\Sigma_j, g'(t)} \left(\frac{\partial g'}{\partial v}_{|s=t}\right) d\Sigma_j
\end{eqnarray*}
for every $v \in \mathbb{R}^N$.
\end{lemme}

\begin{proof}

Let $g:I^N \rightarrow \Gamma_q$ be a smooth embedding.  Consider a sequence $\{S_i\}_i$ that enumerates all the diffeomorphism types of $n$-dimensional closed manifolds, and let $\mathfrak{M}(S_i)$ be the Banach manifold of pairs $(\gamma, [u])$ as in the Structure Theorem of White \cite{white2} (Theorem 2.1), where $\gamma$ is a $C^q$ Riemannian metric and $u:S_i \to M$ is a $C^{2,\alpha}$ embedding that is minimal with respect to $\gamma$.  Define $\mathfrak{M}:=\bigcup_i \mathfrak{M}(S_i)$ and the projection $\Pi:\mathfrak{M}\to \Gamma_q$ which sends $(\gamma,[u])$ to $\gamma$. Theorem 2.1 of \cite{white2} (see also \cite{white}) gives that $\mathfrak{M}$ is a separable $C^{q-2}$ Banach manifold and that $\Pi$ is a $C^{q-2}$ Fredholm map with Fredholm index zero. The pair $(\gamma, [u])$ is a critical point of $\Pi$ if and only if $u$ admits a nontrivial Jacobi field with respect to the metric $\gamma$.

We can perturb $g:I^N \rightarrow \Gamma_q$ slightly in the $C^\infty$ topology to a $C^\infty$ embedding $g':I^N \rightarrow \Gamma_q$ that is transversal to $\Pi: \mathfrak{M} \rightarrow \Gamma_q$ by Smale's Transversality Theorem (Theorem 3.1 of  \cite{smale}). Transversality implies $\tilde{I}^N = \Pi^{-1}(g'(I^N))$ is an $N$-dimensional submanifold of $\mathfrak{M}$ (Theorem 3.3 of \cite{smale}).
Let $\pi=(g')^{-1} \circ \Pi_{|\tilde{I}^N}$, so $\pi: \tilde{I}^N \rightarrow I^N$. Let $\mathcal{A}'$ be the subset of points $t' \in I^N$ which are regular values of $\pi$ and such that the Lipschitz function $t \mapsto \omega_p(t):= \omega_p(M,g'(t))$ is differentiable at $t'$ for all $p$. This subset is of full Lebesgue measure in $I^N$ by Rademacher's Theorem and Sard's Theorem. 
Note that if $t' \in \mathcal{A}'$, then $g'(t')$ is a regular value of $\Pi$. 

%In particular, for any $c>0$, the set of one-sided, connected, closed, smooth embedded minimal hypersurfaces $\Sigma \subset (M,g'(t'))$ %with $|A|\leq c$ and ${\rm vol}_{g'(t')}(\Sigma) \leq c$ is finite,  where $|A|$ is the norm of the second fundamental form.

Consider 
\begin{eqnarray*}
&&\mathcal{S}'_{\kappa,c}(p) := \Big\{t\in I^N; \exists V\in \mathcal{V}(g'(t)), ||V||_{g'(t)}(M)=\omega_p(t),\\
&& \hspace{2cm} {\rm index}(\spt(V))\leq p, \max_{\spt(V)}|A| \leq \kappa,\\
&& \hspace{2cm} {\rm \, and \,} V {\rm \, satisfies\,} (\star)_{c, \kappa, \sup_{s\in I^N}\omega_p(s)}\Big\}
\end{eqnarray*}
for any $\kappa >0$ and $c>0$, where 
\begin{itemize}
\item[($\star_{c,\kappa, a}$)]: every two-sided, connected component of  $\spt{V}$,  has varifold distance (in the metric $g'(0)$) at least $c$ of any varifold $2 \Sigma$, where $\Sigma$ is a one-sided, embedded, connected minimal hypersurface in $g'(s)$ with $|A|\leq \kappa$ and ${\rm vol}_{g'(s)}(\Sigma) \leq a$ for some $s\in I^N$.
\end{itemize}

%change for rationals
Each $\mathcal{S}'_{\kappa,c}(p)$ is a closed set, by convergence properties of minimal hypersurfaces. Proposition 2.2 of \cite{irie-marques-neves} (which uses the index estimates of \cite{marques-neves-index}) also holds for $C^q$ metrics if we allow the minimal hypersurfaces to be $C^2$. This follows, for instance, by approximating the $C^q$ metric by $C^\infty$ metrics, applying Proposition 2.2 of \cite{irie-marques-neves} to these metrics and using Sharp's Compactness Theorem (\cite{sharp}). It implies $$
\cup_{\kappa, c \in \mathbb{Q}_+} \,\mathcal{S}'_{\kappa,c}(p) = I^N
$$
for every $p \in \mathbb{N}$.

For $\kappa>0$ and $c>0$, define $\mathcal{S}_{\kappa,c}(p)$ to be the set of points where the Lebesgue density of $\mathcal{S}'_{\kappa,c}(p)$ is one. By the Lebesgue density theorem, $\mathcal{S}'_{\kappa,c}(p)\backslash \mathcal{S}_{\kappa,c}(p)$ has measure zero. Finally, define the full Lebesgue measure set
$$\mathcal{A}:=\mathcal{A}'\cap \bigcap_p \bigcup_{\kappa, c}\mathcal{S}_{\kappa,c}(p).$$

Fix $p\in \mathbb{N}$, and let $t\in\mathcal{A}$. There exist  $\kappa>0$ and $c>0$ such that $t\in\mathcal{S}_{\kappa,c}(p)$. Since the Lebesgue density of  $\mathcal{S}'_{\kappa,c}(p)$ at $t$ is one, we have that   for any unit direction $v$, there is a sequence $\{t_m(v)\}_m\subset \mathcal{S}'_{\kappa,c}(p)$ converging to $t$ with $\frac{t_m(v)-t}{|t_m(v)-t|}$ converging to $v$, so that 
\begin{equation} \label{quick}
\lim_{m\to \infty} \frac{\omega_p(t_m(v))-\omega_p(t)}{|t_m(v)-t|} = \frac{\partial}{\partial v} \omega_p(t).
\end{equation}

Fix $v$ and a corresponding sequence $\{t_m(v)\}_m$. By construction, for each $m$ there is a $V_m\in\mathcal{V}(g'(t_m(v)))$ with mass $\omega_p(t_m(v))$, with ${\rm index}(\spt(V_m))\leq p$, whose support has second fundamental form bounded by $\kappa$ (which is independent of $m$) and such that every two-sided, connected component of  $\spt{V_m}$  has varifold distance (in the metric $g'(0)$) at least $c$ (also independent of $m$) from any varifold $2 \Sigma$, where $\Sigma$ is a one-sided, connected minimal hypersurface in $g'(s)$ with $|A|\leq \kappa$ and ${\rm vol}_{g'(s)}(\Sigma) \leq \sup_{s\in I^N}\omega_p(s)$ for some $s\in I^N$. This implies no two-sided component of $\spt{V_m}$ can collapse, after maybe passing to a subsequence, to a one-sided component with multiplicity two. 
Choosing a subsequence and renumbering if necessary, $V_m$ converges to a varifold $V\in\mathcal{V}(g'(t))$ and the supports $\spt(V_m)$ converge in $C^2$ to $\spt(V)$. This convergence is with multiplicity one, because if not one could construct by a standard argument a nontrivial Jacobi field on one of the components of $\spt(V)$. This is not possible, since  $g'(t)$ is a regular value of $\Pi$.

Consider a sequence $\{\Sigma_m\}$ of  connected components of $\spt(V_m)$ that converges in $C^2$ to $\Sigma$. By elliptic regularity, the convergence is also in $C^{2,\alpha}$. The corresponding points $$\tilde{z}_m = (g'(t_m(v)), [\Sigma_m]) \in \tilde{I}^N \subset \mathfrak{M}$$  converge to a point $z\in\Pi^{-1}(g'(t)) \subset \tilde{I}^N$, $z=(g'(t),[\Sigma])$. Note that since $t\in\mathcal{A}$, $\pi$ is  a local diffeomorphism from a neighborhood of $z$ in $\tilde{I}^N$ to a neighborhood of $t$ in $I^N$. We write
$\tilde{z}=(g'(\pi(\tilde{z})), [\Sigma(\pi(\tilde{z}))])$ for any $\tilde{z}$ in this neighborhood of $z$. For sufficiently large $m$, $[\Sigma_m]=[\Sigma(t_m(v))]$. But then
\begin{eqnarray*}
&&\lim_{m\to \infty} \frac{{\rm vol}_{g'(t_m(v))}(\Sigma(t_m(v)))-{\rm vol}_{g'(t)}(\Sigma(t))}{|t_m(v)-t|} = \frac{\partial}{\partial v} {\rm vol}_{g'(s)}(\Sigma(s))_{|s=t}\\
&&=\frac12 \int_\Sigma \Tr_{\Sigma, g'(t)}\left(\frac{\partial g'}{\partial v}(t)\right) d \Sigma.
\end{eqnarray*}
Taking into account the multiplicity of each connected component of $\spt(V_m)$, the limit in (\ref{quick}) becomes
$$\frac{\partial}{\partial v} \omega_p(t)= \int_{V} \frac12 \Tr_{V, g'(t)} \left(\frac{\partial g'}{\partial v}_{|s=t}\right) d||V||(M),$$
where $V$ is of the form $\sum_{i=1}^Q m_i\tilde{\Sigma}_i$, with $\{\tilde{\Sigma}_1, \dots, \tilde{\Sigma}_Q\}$ a disjoint collection of closed, $C^{2,\alpha}$, embedded, minimal hypersurfaces in $(M,g'(t))$ and $\{m_1, \dots, m_Q\} \subset \mathbb{N}$, $||V||(M)=\omega_p(t)$, $\sum_{i=1}^Q {\rm index}(\tilde{\Sigma}_i) \leq p$, $\max_{\spt(V)}|A| \leq \kappa$ and $V$ satisfies $(\star_{\kappa, c, \sup_{s\in I^N}\omega_p(s)})$. 
%change second fundamental form bounded
By elliptic regularity, each $\tilde{\Sigma}_i$ is of class $C^q$.
Since $t$ is a regular value of $\pi$, every embedded minimal hypersurface of $(M,g'(t))$ is non-degenerate. Because convergence of the supports can only happen with multiplicity one, there are only finitely many
%correction1 
$V$'s as above, say $\{V^{(1)}, \dots, V^{(P)}\}$. For any unit direction $v \in \mathbb{R}^N$, one has
$$\frac{\partial}{\partial v} \omega_p(t)= \int_{V^{(l)}} \frac12 \Tr_{V^{(l)}, g'(t)} \left(\frac{\partial g'}{\partial v}_{|s=t}\right) d||V^{(l)}||(M)$$
for some $1\leq l \leq P$. This means that there will be a single $1\leq l \leq P$ such that the above formula is true for a linearly independent set $\{v_1,\dots, v_N\}$, and hence for every $v$ by linearity. This finishes the proof.

\end{proof}

The next lemma concerns the gradient of Lipschitz functions that are almost constant. The convex hull of a set $K\subset \mathbb R^N$ is denoted by  $\Conv(K)$.
\begin{lemme} \label{maximum point}
Given $\delta>0$ and $N \in \mathbb{N}$, there exists $\varepsilon>0$ depending on $\delta$ and $N$ such that the following is true:  for  any Lipschitz function $f:I^N \to \mathbb{R}$  satisfying
$$
|f(x)-f(y)|\leq 2\varepsilon
$$
for every $x,y\in I^N$, and for any subset $\mathcal{A}$ of $I^N$ of full measure,
 there exist $N+1$ sequences of points $\{y_{1,m}\}_m,\dots,\{y_{N+1,m}\}_m$ contained in $\mathcal{A}$ and converging to a common limit $y\in(0,1)^N$ such that:
\begin{itemize}
\item $f$ is differentiable at each $y_{i,m}$,
\item the gradients $\nabla f(y_{i,m})$ converge to $N+1$ vectors $v_1,\dots,v_{N+1}$ with
$$
d_{\mathbb{R}^N}\left(0, \Conv(v_1,\dots,v_{N+1})\right) < \delta,
$$ 
\end{itemize}

\end{lemme}

\begin{proof}
Suppose, by contradiction, that the lemma is false. Then there exists a sequence of Lipschitz functions $f_k:I^N \to \mathbb{R}$  satisfying
$$
|f_k(x)-f_k(y)|\leq 1/k
$$
for every $x,y\in I^N$, and a sequence of sets $\mathcal{A}_k \subset I^N$ of full measure, such that these sequences of points do not exist. Since $f_k$ is Lipschitz, the set $\mathcal{D}_k\subset I^N$ of points where $f_k$ is differentiable has full measure by Rademacher's Theorem. Hence the set $\mathcal{A}'_k=\mathcal{A}_k \cap \mathcal{D}_k$ has full measure also.

Choose a smooth function $g:I^N \to \mathbb{R}$ such that $g$ is equal to $1$ on the boundary of $I^N$ and equal to $0$ at $(1/2,\dots,1/2)\in I^N$. Then  the Lipschitz function $h_k=f_k-\frac{2}{k}g$ achieves its maximum at an interior point $y_k\in(0,1)^N$. Consider the  set $V_k \subset \mathbb{R}^N$ of vectors $v$ such that there exists a sequence $z_m \in \mathcal{A}'_k$ with $z_m \rightarrow y_k$ and $\nabla h_k(z_m)\rightarrow v$ as $m \rightarrow \infty$. The set $V_k$ is bounded and closed. For almost all directions $w$ in the unit sphere $S^{N-1}$, the set $$\{t\in [0,d_{\mathbb{R}^N}(y_k,\partial I^N)]: y_k+t w \in \mathcal{A}_k'\}$$
has full measure in $ [0,d_{\mathbb{R}^N}(y_k,\partial I^N)]$. For any such $w$, because $h_k$ has a maximum point at $y_k$, there exists $v \in V_k$ with $\langle v, w\rangle \leq 0$. This implies that for any $w \in \mathbb{R}^N$, there exists $v\in  V_k$ with $\langle v, w\rangle \leq 0$.  By the Hahn-Banach Theorem, $0 \in {\rm Conv}(V_k)$. Caratheodory's Theorem gives vectors $\{\tilde{v}_1, \dots, \tilde{v}_{N+1}\} \subset V_k$ such that $0 \in {\rm Conv}(\{\tilde{v}_1, \dots, \tilde{v}_{N+1}\})$. Hence  there exist $N+1$ sequences of points $\{y^{(k)}_{1,m}\}_m,\dots,\{y^{(k)}_{N+1,m}\}_m$ contained in $\mathcal{A}_k'$ and converging to  $y_k\in(0,1)^N$ such that:
\begin{itemize}
\item $f_k$ is differentiable at each $y^{(k)}_{i,m}$,
\item the gradients $\nabla f_k(y^{(k)}_{i,m})$ converge to $N+1$ vectors $v_1,\dots,v_{N+1}$ with
$$
d_{\mathbb{R}^N}\left(0, \Conv(v_1,\dots,v_{N+1})\right) \leq \frac{2}{k} \sup_{I^N} |\nabla g|.
$$ 
\end{itemize}

If $k$ is sufficiently large, $\frac{2}{k} \sup_{I^N} |\nabla g| < \delta$. Contradiction.
\end{proof}

%correction3: adding the last lemma 
%change to suppose union of embedded hypersurfaces, and to use unique continuation for eigenfunctions
The last lemma shows that one can  make  finitely many closed, embedded, minimal hypersurfaces nondegenerate by an arbitrarily small conformal change of the metric. This  generalizes Proposition 2.3 of \cite{irie-marques-neves}.
\begin{lemme} \label{make hypersurface nondegenerate again}
Suppose $g\in \Gamma_q$, $q\geq 2$. Let $\{\Sigma_1, \dots, \Sigma_L\}$ be a finite collection of  closed, embedded, connected, $C^2$ minimal hypersurfaces in $(M,g)$. Then there exists a sequence of metrics $g_i\in \Gamma_q$, $i\in \mathbb{N}$, converging to $g$ in the $C^q$ topology such that $\Sigma_j$ is a nondegenerate minimal hypersurface in $(M,g_i)$ for all $j=1, \dots, L$ and $i\in \mathbb{N}$.
\end{lemme}

\begin{proof}
Each $\Sigma_i$ is $C^q$ by elliptic regularity. We can suppose $\Sigma_j \neq \Sigma_k$ when $j \neq k$. Choose $\delta>0$ such that $B_r(q) \cap \Sigma_j$ is connected for every $j=1, \dots, L$, $0<r\leq \delta$ and $q \in \Sigma_j$. We claim that there exists a point $p \in \Sigma_1 \setminus (\Sigma_2 \cup \cdots \cup \Sigma_L)$.

Pick $x_1 \in \Sigma_1$ arbitrary. If $x_1 \notin \Sigma_2$, set $x_2=x_1$. Suppose $x_1 \in \Sigma_2$. If $B_\delta(x_1)\cap \Sigma_1  \subset \Sigma_2$, then $\Sigma_1=\Sigma_2$ by unique continuation. This is not possible, hence there exists $x_2 \in B_\delta(x_1)\cap \Sigma_1$ but $x_2 \notin \Sigma_2$. In any case, we have found $x_2 \in \Sigma_1 \setminus \Sigma_2$. Suppose we have $x_j \in \Sigma_1 \setminus (\Sigma_2 \cup \dots \cup \Sigma_j)$, $2 \leq j \leq L-1$. If $x_j \notin \Sigma_{j+1}$, set $x_{j+1}=x_j$. Assume $x_j \in \Sigma_{j+1}$, and define $\delta_j = \min \{\delta, \frac12 d(x_j, \Sigma_2 \cup \dots \cup \Sigma_j)\}>0.$ If $B_{\delta_j}(x_j) \cap \Sigma_1 \subset \Sigma_{j+1}$, then $\Sigma_1=\Sigma_{j+1}$ by unique continuation. This is impossible, hence there exists $x_{j+1} \in B_{\delta_j}(x_j) \cap \Sigma_1$ but $x_{j+1} \notin \Sigma_{j+1}$. In any case, we have found $x_{j+1} \in \Sigma_1 \setminus (\Sigma_2 \cup \dots \cup \Sigma_{j+1})$. By induction, we find $x_L \in \Sigma_1 \setminus (\Sigma_2 \cup \cdots \cup \Sigma_L)$.

For similar reasons, there exist $p_l \in \Sigma_l \setminus (\cup_{k\neq l} \Sigma_k)$ for every $l=1, \dots, L$.
Choose $\eta>0$ sufficiently small so that $\eta$ is smaller than the injectivity radius of the manifold and such that $\eta<\frac{1}{4}d_g(p_l, \cup_{k\neq l} \Sigma_k)$ for every $l$. By decreasing $\eta$ if necessary, we can choose for each $l=1, \dots, L$, a $C^q$ function $f_l:B_\eta(p_l) \rightarrow \mathbb{R}$ such that $f_l=0$  and $\langle \nabla f_l, N_l\rangle >0$ on $\Sigma_l \cap B_\eta(p_l)$, where $N_l$ is a local choice of unit normal to $\Sigma_l$. We also choose, for each  $l=1, \dots, L$, a smooth nonnegative function $\varphi_l:M \rightarrow \mathbb{R}$ such that $\varphi_l=1$ on $B_{\eta/2}(p_l)$ and $\varphi_l=0$ outside $B_{2\eta/3}(p_l)$.

%Let $\eta_l:M \rightarrow \mathbb{R}$ be a smooth function such that is equal to 1 in $V_\eta(\Sigma_l)$ and equal to zero in $M \setminus V_{2\eta}
%(\Sigma_l)$, where $V_r(\Sigma_l) = \{x\in M: d_g(x,\Sigma_l) \leq r\}$. We choose $\eta>0$ sufficiently small so that the function $x \mapsto d_g(x,
%\Sigma_l)^2$ is $C^q$ in $V_{3\eta}(\Sigma_l)$,  the nearest-point projection $\pi_l:V_{3\eta}(\Sigma_l) \rightarrow \Sigma_l$ is well-defined and of %class $C^q$, and such that $\eta<\frac{1}{100}d_g(p_l, \cup_{k\neq l} \Sigma_k)$. Choose a $C^q$ function $\varphi_l:\Sigma_l \rightarrow [0,1]$ that %is equal to 1 on $B_\eta(p_l) \cap \Sigma_l$ and equal to 0 on $\Sigma_l \setminus B_{2\eta}(p_l)$. We define 
%$h_l(x) = -\eta_l(x) \varphi_l(\pi_l(x)) d_g(x,\Sigma_l)^2$ for $x\in V_{3\eta}(\Sigma_l)$ and $h_l(x)=0$ for $x \in M \setminus V_{3\eta}(\Sigma_l)$. %Note that $h_l(x)=0$ for every $x \in M \setminus B_{6\eta}(p_l)$ and that $h_l(x) =  -\varphi_l(\pi_l(x)) d_g(x,\Sigma_l)^2$ in a small neighborhood of $%\Sigma_l \cap B_{7\eta}(p_l)$. 

Let $g_i = \exp(2\phi_i)g$, where $\phi_i=-\frac{1}{i}(\varphi_1f_1^2 + \cdots + \varphi_Lf_L^2)$. By the arguments  of \cite[Proposition 2.3]{irie-marques-neves}, one can check that at any point $y$ on $\Sigma_l$, $\phi_i=0$, $\nabla\phi_i=0$ and $\Hess_g\phi_i(N,N)= -\frac{2}{i}\varphi_l(y)\langle \nabla f_l, N\rangle^2(y)$, where $N$ is a unit normal to $\Sigma_l$ at $y$ with respect to the metric $g$ (or $g_i$). This implies $\Sigma_l$ remains minimal with respect to  $g_i$ for every $l$, and at  points of $\Sigma_l$ we have:
$$\Ric_{g_i}(N,N)+|A_{\Sigma_l,g_i}|^2_{g_i} = \Ric_{g}(N,N)+|A_{\Sigma_l,g}|^2_{g}+ \frac{2n}{i}\varphi_l\langle \nabla f_l, N\rangle^2,$$
where $|A_{\Sigma_l,g}|$ is the norm of the second fundamental form of $\Sigma_l$ with respect to $g$.

The Jacobi operator of $\Sigma_l$ acting on normal vector fields is given by
$$L_{\Sigma_l,g}(X) = \Delta^\perp_{\Sigma_l,g}X+(\Ric_g(N,N)+|A_{\Sigma_l,g}|^2_g)X.$$
 Since $g_i$ and $g$ coincide on $\Sigma$, $\Delta^\perp_{\Sigma_l,g}=\Delta^\perp_{\Sigma_l,g_i}$ and hence
\begin{equation} \label{shift}
L_{\Sigma_l,g_i}(X)=L_{\Sigma_l,g}(X)+\frac{2n}{i}\varphi_l \langle \nabla f_l, N\rangle^2X.
\end{equation}

Fix $l$, and define $\tilde{L}_t(X)=L_{\Sigma_l,g}(X) + t\varphi_l\langle \nabla f_l, N\rangle^2X$ on $\Sigma_l$, for $t \in \mathbb{R}$. It is known that the eigenvalues of $\tilde{L}_t$ depend continuously on the parameter $t$. Suppose that $\Sigma_l$ is  a degenerate minimal hypersurface in $(M,g)$, and let $Q$ be the unique integer such that $0=\lambda_Q(\tilde{L}_0)<\lambda_{Q+1}(\tilde{L}_0)$. If $t$ is sufficiently small, then $\lambda_{Q+1}(\tilde{L}_t)>0$.

Let $X$ be in the zero eigenspace $E$ of $\tilde{L}_0$, $X \neq0$. Then 
\begin{eqnarray*}
&&\frac{d}{dt}_{|t=0}\left(-\frac{\int_{\Sigma_l} \langle \tilde{L}_t(X),X\rangle}{\int_{\Sigma_l}|X|^2}\right) \\
&&=\frac{d}{dt}_{|t=0}\left( \frac{-\int_{\Sigma_l}\langle \tilde{L}_0(X),X \rangle -t\int_{\Sigma_l} \varphi_l\langle \nabla f_l, N\rangle^2 |X|^2}{\int_{\Sigma_l} |X|^2}\right) \\
 && =  - \frac{\int_{\Sigma_l} \varphi_l\langle \nabla f_l, N\rangle^2 |X|^2}{\int_{\Sigma_l}|X|^2} \\
 &&\leq- \frac{\int_{B_{\eta/2}(p_l) \cap \Sigma_l}\langle \nabla f_l, N\rangle^2|X|^2}{\int_{\Sigma_l}|X|^2}.
\end{eqnarray*}
By unique continuation of solutions of linear elliptic equations and the finite-dimensionality of $E$, we can find a constant $c>0$ such that
\begin{equation}\label{derivative.eigenvalue}
\frac{d}{dt}_{|t=0}\left(-\frac{\int_{\Sigma_l} \langle \tilde{L}_t(X),X\rangle}{\int_{\Sigma_l}|X|^2}\right)  \leq -c
\end{equation}
for every $X \in E\setminus \{0\}$.

Recall the min-max characterization of the eigenvalue $\lambda_Q(\tilde{L}_t)$:
\begin{equation} \label{mmcharacterization}
\lambda_Q(\tilde{L}_t) = \inf_W \max_{X\in W\backslash\{0\}} \frac{-\int_{\Sigma_l} \langle \tilde{L}_t(X),X \rangle}{\int_{\Sigma_{l}} |X|^2},
\end{equation}
where the infimum is taken over all the $Q$-dimensional subspaces $W$ of the space of smooth, normal vector fields on $\Sigma_l$. If $\tilde{W}$ is the subspace spanned by the eigensections of $\tilde{L}_0$ corresponding to eigenvalues $\lambda \leq 0$, then ${\rm dim}(W)=Q$. By combining (\ref{derivative.eigenvalue}) and (\ref{mmcharacterization}), we have
$$
\lambda_Q(\tilde{L}_t) \leq \max_{X\in \tilde{W}\backslash\{0\}} \frac{-\int_{\Sigma_l} \langle \tilde{L}_t(X),X \rangle}{\int_{\Sigma_{l}} |X|^2}\leq -\frac{c}{2}t
$$
for sufficiently small $t \geq 0$. Therefore for sufficiently large $i$ we have both $\lambda_Q(L_{\Sigma_l,g_i})<0$ and $\lambda_{Q+1}(L_{\Sigma_l,g_i})>0$. This implies $\Sigma_l$ is nondegenerate with respect to $(M,g_i)$ for sufficiently large $i$. Since this is true for every $l=1, \dots, L$, the Lemma is proved.

\end{proof}

\section{Proof of the Main Theorem}

Let $g$ be a smooth Riemannian metric on $M$, $K$ be an integer and $\epsilon_1>0$ be a positive constant 
%correction1: so that parallel transport will be well-defined in B_k
smaller than the injectivity radius of $g$. Let  $\hat{B}_1,\dots,\hat{B}_K$ be disjoint domains in $M$, with piecewise smooth boundary, such that the union of their closures covers $M$.
%correction1: I put this conidition on B_k instead
%the closure of each $\hat{B}_k$ is contained in a geodesic ball of radius $\epsilon_1$.
%correction1: I removed this condition on the volume since we only use the fact that each B_k is in a small ball, not that they have almost the same volume
%and $$\sum_{k=1}^K |{\rm vol}_g(\hat{B}_k) - \frac{{\rm vol}_g(M)}{K}| < \epsilon_1.$$ 

%correction1
%For each $k\in\{1,\dots,K\}$, let $B_k'$ be a domain with $\overline{B'_k}\subset \hat{B}_k$ with ${\rm vol}_g(B_k') \geq {\rm vol}_g(\hat{B}_k)-\epsilon_1/K$. 
Let $B_k$ be some neighborhood of $\hat{B}_k$. We suppose that each $B_k$ is contained in 
%correction1
a geodesic ball of radius $\epsilon_1$. 
%the $\epsilon_1$-neighborhood of $\hat{B}_k$ 
%correction1
%and such that ${\rm vol_g}(B_k) < {\rm vol_g}(\hat{B}_k)+\epsilon_1/K$. 
%We can choose these domains so that  $B_k \cap B_q' = \emptyset$ whenever $k \neq q$. 
Choose also a smooth function $0\leq \phi_k \leq 1$ that is equal to $1$ on $\hat{B}_k$ and with $\spt(\phi_k) \subset B_k$, and a point $q_k\in \hat{B}_k$ for each $k$. We can also suppose that $q_k \notin B_l$ if $l\neq k$. Define the partition of unity $\psi_k=\frac{\phi_k}{\sum_q\phi_q}$. Hence $\psi_k(q_k)=1$ and $\psi_k(q_l)=0$ for $l \neq k$.

For a fixed $k$, let $e$ be a unit vector in the tangent space of $M$ at $q_k$. It determines by parallel transport along geodesics starting at $q_k$ a unit vector field in $B_k$ still denoted by $e$. We define a  nonnegative symmetric $(0,2)$-tensor $h(e)$ on $B_k$ as follows: 
$h(e)(v,w) =
%correction1: added g
 \langle v,e\rangle_g \langle w,e\rangle_g$.

Now consider the space $\mathcal{B}_k$ of orthonormal bases at $q_k$; these $\mathcal{B}_k$ are endowed with a natural metric 
%correction1: 
determined by $g$
and of course are isometric to each other. For each $k$, pick $L$ points $x^k_1,\dots,x^k_L\in\mathcal{B}_k$ such that any point in $\mathcal{B}_k$ is at distance less than 
%correction1: \epsilon_1 instead of \epsilon_1/K is enough, see (small) and (small2) below
$\epsilon_1$ to one of the $x^k_l$. Each $x^k_l$ is an orthonormal basis $(x^k_{l,1},...,x^k_{l,n+1})$ at $q_k$ and so we can consider the family of symmetric $(0,2)$-tensors $h^k_{l,j}=h(x^k_{l,j})$. Note that by construction, in  $B_k$, for any $l$ the sum $\sum_{j=1}^{n+1} h^k_{l,j}$ is the metric $g$. 

We denote by $\mathcal{C}_{g, \tilde{K}, \epsilon_1}$  the set of all possible choices $$(K, \{\hat{B}_k\}, \{B_k\}, \{\phi_k\}, \{q_k\}, \{x^k_l\})$$ as above, with $K \geq \tilde{K}$.  
%correction1
The set $\mathcal{C}_{g, \tilde{K}, \epsilon_1}$ is non-empty, as can be seen by taking a sufficiently fine triangulation of $M$. 
%In order to see this, choose a triangulation $T = \{s_1, \dots, s_N\}$ of $M$ such that each simplex $s_i$ is contained in $B_{\varepsilon}(p_i)$ for some $p_i \in M$, with $\varepsilon = \min\{\epsilon_1, {\rm inj}(M,g)/2\}$. Here ${\rm inj}(M,g)$ is the injectivity radius of $(M,g)$. Choose integers $a_i, b_i \in \mathbb{N}$, such that
%$$|{\rm vol}_g(s_i)-\frac{a_i}{b_i}| < \frac{\epsilon_1}{2N}.$$
%We can suppose $b_1=\cdots=b_N=b$, and $a_i \geq \tilde{K}$. Since ${\rm vol}_g(M)=\sum_{i=1}^N {\rm vol}_g(s_i)$,
%$$|{\rm vol}_g(M)-\frac{\sum_{i=1}^Na_i}{b}| < \frac{\epsilon_1}{2}.$$
%For each $i$, we can decompose $s_i$ into disjoint subdomains $\tilde{s}_{i,1}, \dots, \tilde{s}_{i, a_i}$ with piecewise smooth boundary, whose closures cover $s_i$ and such that ${\rm vol}_g(\tilde{s}_{i,j})={\rm vol}_g(s_i)/a_i.$ Then
%\begin{eqnarray*}
%&&\sum_{i=1}^N \sum_{j=1}^{a_i}|{\rm vol}_g(\tilde{s}_{i,j})-\frac{{\rm vol}_g(M)}{\sum_{i=1}^Na_i}|=\sum_{i=1}^Na_i |\frac{{\rm vol}_g(s_{i})}{a_i}-\frac{{\rm vol}_g(M)}{\sum_{i=1}^Na_i}|\\
%&&=\sum_{i=1}^N|{\rm vol}_g(s_{i})-\frac{a_i{\rm vol}_g(M)}{\sum_{i=1}^Na_i}|< \sum_{i=1}^N|{\rm vol}_g(s_{i})-\frac{a_i}{b}|+\frac{\epsilon_1}{2}<\epsilon_1.
%\end{eqnarray*}
%We choose $K=\sum_{i=1}^N a_i$ and $\{\hat{B}_k\}=\{\tilde{s}_{i,j}\}$. The above estimate implies $\mathcal{C}_{g, \tilde{K}, \epsilon_1}$ is non-empty.}

Recall that $\mathcal{V}(g)$ denotes  the set of stationary integral varifolds in $(M,g)$ whose support is an embedded minimal hypersurface. We claim that in order to show the main theorem, it suffices to prove the following property.

\medskip

(\textbf{P}): For any metric $g$, for every $\epsilon_1>0$, $\tilde{K}>0$ and any choice of 
$$S=(K, \{\hat{B}_k\}, \{B_k\}, \{\phi_k\}, \{q_k\}, \{x^k_l\}) \in \mathcal{C}_{g, \tilde{K}, \epsilon_1},$$   there is a metric $\tilde{g}$ arbitrarily close to $g$ in the $C^\infty$ topology such that there are varifolds $V_1,\dots,V_J$ of $\mathcal{V}(\tilde{g})$ 
%correction2: added nondegeneracy here, will be proved at the end 
whose support $\spt(V_j)$ are nondegenerate, and coefficients $\alpha_1,\dots,\alpha_J\in[0,1]$ with $\sum_i \alpha_i =1$ satisfying
\begin{equation} \label{sufficient}
\forall k,l,j \quad \Big| \sum_i \alpha_i \frac{V_i(\psi_k h^k_{l,j})}{||V_i||(M)} - 
%correction1: precised it is computed with \tilde{g} instead of g, the reason for this choice is that the V() are always computed with \tilde{g}, so it seems less confusing to also compute the integrals with \tilde{g}
\frac{1}{(n+1)}\frac{1}{{\rm vol}_{\tilde{g}}(M)}\int_M \psi_k dv_{\tilde{g}}\Big| < \epsilon_1/K,
\end{equation}
where the terms of the sum are computed for the metric $\tilde{g}$.  Here $$V(h)=\int_{G_n(M)} h(\nu,\nu) dV(p, \pi),$$ where $G_n(M)$ denotes the Grassmannian of $n$-dimensional planes of $M$ and $\nu$ is a unit normal to the $n$-plane $\pi \subset T_pM$.
\medskip

%correction2: reformulated this sentence
Indeed, let us explain why Property ({\bf P}) implies the main theorem. We denote by $\mathcal{M}(g,\epsilon_1,\tilde{K}, S)$, with $$S=(K, \{\hat{B}_k\}, \{B_k\}, \{\phi_k\}, \{q_k\}, \{x^k_l\})\in \mathcal{C}_{g, \tilde{K}, \epsilon_1},$$ the family of metrics $\tilde{g}\inÊ\Gamma_\infty$ at distance less than $\epsilon_1/K$ to $g$ (computed with respect to $g$) 
%correction1
%in the $C^\infty$ topology 
in the $C^K$ topology such that there are $\{V_1,\dots,V_J\} \subset \mathcal{V}(\tilde{g})$ 
%correction3: added "whose support are nondegenerate"
whose supports are nondegenerate, $\alpha_1,\dots,\alpha_J\in[0,1]$ with $\sum_i \alpha_i =1$, which satisfy (\ref{sufficient}) for all $k, l, j$. If $g' \in \mathcal{M}(g,\epsilon_1,\tilde{K}, S)$, and $\{\Sigma_1', \dots, \Sigma_Q'\}$ is any finite collection of nondegenerate minimal hypersurfaces in $(M,g')$,  then for every metric $\tilde{g}$ that is sufficiently close to $g'$, there is a unique collection $\{\tilde{\Sigma}_1, \dots, \tilde{\Sigma}_Q\}$ of nondegenerate minimal hypersurfaces in $(M,\tilde{g})$ such that $\tilde{\Sigma}_i$ is close to $\Sigma_i'$, $i=1, \dots, Q$. Moreover, $\tilde{\Sigma}_i$ converges smoothly to $\Sigma_i'$ as $\tilde{g}$ converges to $g'$. This implies that $\mathcal{M}(g,\epsilon_1,\tilde{K}, S)$ is open in the $C^\infty$ topology.

Define $$\mathcal{M}(\epsilon_1,\tilde{K}) : = \bigcup_{g\in \Gamma_\infty}\bigcup_{S \in \mathcal{C}_{g, \tilde{K}, \epsilon_1}} \mathcal{M}(g,\epsilon_1,\tilde{K},S).$$
 It is clearly open. Given an arbitrary metric $g \in \Gamma_\infty$, we can choose $S \in  \mathcal{C}_{g, \tilde{K}, \epsilon_1}$.
%correction1: make the hypersurfaces nondegenerate 
%By \cite[Proposition 2.3]{irie-marques-neves}, 
%correction2: we don't need this Prop 2.3 that I added in correction1, we will prove a generalization for immersed surfaces later
Property ({\bf P}) implies that the metric $g$ is a limit of metrics in $\mathcal{M}(g,\epsilon_1,\tilde{K},S)$. This shows that $\mathcal{M}(\epsilon_1,\tilde{K})$ is also dense.

Define
$$\mathcal{M}:= \bigcap_{m\in \mathbb{N}} \mathcal{M}(1/m,m).$$
Since each $ \mathcal{M}(1/m,m)$ is open and dense, the intersection $\mathcal{M}$ is a residual subset (in the Baire sense)  of the set of metrics. We will show that for any metric in $\mathcal{M}$, one can find sequences
%correction2: no need of finite lists
of minimal hypersurfaces like in the Main Theorem.
For any metric, a symmetric $(0,2)$-tensor $h$ is diagonalizable at every point. The idea is to find a fine subdivision of $M$ in domains $B_k$ where $h$ is approximately diagonal when expressed in the basis $x^k_{l(k)}$ for a certain $l(k)\in\{1,\dots,L\}$. 

Let $\tilde{g} \in \mathcal{M}$. Then $\tilde{g}\in\mathcal{M}(1/m,m)$ for every $m \in \mathbb{N}$. Fix $m$. Then by construction there exists a metric $g$ such that $\tilde{g} \in  \mathcal{M}(g,1/m,m,S)$ for some choice of 
$$S=(K, \{\hat{B}_k\}, \{B_k\}, \{\phi_k\}, \{q_k\}, \{x^k_l\}) \in \mathcal{C}_{g, m, 1/m}.$$ 
In particular, $g$ belongs to a $1/(mK)$-neighborhood of $\tilde{g}$ 
%correction1
%in the $C^\infty$ topology. 
in the $C^K$ topology. We also have $\{V_1,\dots,V_J\} \subset \mathcal{V}(\tilde{g})$, $\alpha_1,\dots,\alpha_J\in[0,1]$ with $\sum_i \alpha_i =1$, which satisfy 
\begin{equation} \label{sufficient2}
\forall k,l,j \quad \Big| \sum_i \alpha_i \frac{V_i(\psi_k h^k_{l,j})}{||V_i||(M)} - 
%correction1: precised it is computed with \tilde{g} instead of g
\frac{1}{(n+1)}\frac{1}{{\rm vol}_{\tilde{g}}(M)}\int_M \psi_k dv_{\tilde{g}}\Big| < 1/(mK).
\end{equation}
Note that $g, S, J, \{V_j\}, \{\alpha_j\}$ all depend on $m$.

Let $h$ be a symmetric $(0,2)$-tensor on $M$. The following computations are done with respect to the metric $\tilde{g}$, unless otherwise specified. We  start by writing 
$$\int_M\Tr(h) = \sum_k \int_{B_k} \psi_k \Tr(h),$$
and
$$\sum_i \alpha_i \frac{V_i(h)}{||V_i||(M)} = \sum_k\sum_i \alpha_i \frac{V_i(\psi_k h)}{||V_i||(M)}. $$

At each $q_k\in \hat{B}_k$, $h$ is diagonalizable 
%correction1: for g not \tilde{g}
for the metric $g$ in a $g$-orthonormal basis $$u^k=(u^k_1,\dots,u^k_{n+1})\in\mathcal{B}_k$$ with eigenvalues $\lambda_1(k),\dots,\lambda_{n+1}(k)$
%correction1: 
and we note that $\sum_j \lambda_j(k)$ is the trace of $h$ at $q_k$ for the metric $g$. Let $l(k)$ be such that $x^k_{l(k)}$ is at distance less than 
%correction1: 1/m instead of 1/(mK)
$1/m$ from $u^k$ in $\mathcal{B}_k$.
%correction1: I added equation  {small2}
We get on $B_k$ (which are contained in balls of radius $1/m$) the following estimates with the metric $\tilde{g}$:
\begin{equation}\label{small}
\Big|h - \sum_{j=1}^{n+1} \lambda_j(k) h(u^k_j)\Big|_{\tilde{g}} < \frac{C}{m},
\end{equation}
\begin{equation}\label{small2}
\Big|\sum_{j=1}^{n+1} \lambda_j(k) h(u^k_j) - \sum_{j=1}^{n+1} \lambda_j(k) h^k_{l(k),j}\Big|_{\tilde{g}} < \frac{C}{m}.
\end{equation}
Here $C$ depends only on $\tilde{g}$ and $h$, and might be different from line to line. 

We have, by (\ref{sufficient2}), that
\begin{eqnarray*} 
&&\forall k \quad \Big| \sum_i \sum_j \alpha_i \frac{V_i(\psi_k \lambda_j(k)h^k_{l(k),j})}{||V_i||(M)} - 
%correction1: \tilde{g} instead of g 
\frac{1}{(n+1)}\frac{1}{{\rm vol}(M)}\int_M (\sum_j\lambda_j(k)\psi_k )\Big| \\
&&\hspace{5cm}< C/(mK).
\end{eqnarray*}
Hence
\begin{eqnarray*} \nonumber
&&\sum_k \Big| \sum_i \sum_j \alpha_i \frac{V_i(\psi_k \lambda_j(k)h^k_{l(k),j})}{||V_i||(M)} -
%correction1: idem
 \frac{1}{(n+1)}\frac{1}{{\rm vol}(M)}\int_M (\sum_j\lambda_j(k)\psi_k) \Big| \\
&&\hspace{5cm}< C/m,
\end{eqnarray*}
%correction1: remark for difference of traces in the two metrics
and since $|\Tr_g h - \Tr_{\tilde{g}} h| <C/m$, we obtain readily
\begin{eqnarray*} \nonumber
&&\sum_k \Big| \sum_i \sum_j \alpha_i \frac{V_i(\psi_k \lambda_j(k)h^k_{l(k),j})}{||V_i||(M)} - 
%correction1: idem
\frac{1}{(n+1)}\frac{1}{{\rm vol}(M)}\int_M \psi_k \Tr h\Big| %\\
%&&\hspace{5cm}
< C/m.
\end{eqnarray*}
%\begin{equation} \nonumber
%\forall k \quad \Big| \sum_i \sum_j \alpha_i \frac{V_i(\psi_k \lambda_j(k)h^k_{l(k),j})}{||V_i||(M)} - \frac{\int_{\hat{B}_k}\Tr h}{(n+1)K{\rm vol}
%(\hat{B}_k)}\Big| < C/(mK).
%\end{equation}
%Now
%\begin{eqnarray*}
%&&\sum_k\Big|\frac{\int_{\hat{B}_k}\Tr h}{(n+1)K{\rm vol}(\hat{B}_k)}- \frac{\int_{\hat{B}_k}\Tr h}{(n+1){\rm vol}(M)}\Big|\\
%&& \leq \sum_k \frac{C |{\rm vol}(M)-K{\rm vol}(\hat{B}_k)|}{(n+1)K{\rm vol}(M)}\leq C/m.
%\end{eqnarray*}
Therefore
\begin{equation} \nonumber
 \Big| \sum_k \sum_i \sum_j \alpha_i \frac{V_i(\psi_k \lambda_j(k)h^k_{l(k),j})}{||V_i||(M)} - 
 %correction1: idem
 \frac{\int_{M}\Tr h}{(n+1){\rm vol}(M)}\Big| < C/m.
\end{equation}

But we also have by (\ref{small}) and (\ref{small2}) that
\begin{equation} \nonumber
 \Big| \sum_k\sum_i \sum_j \alpha_i \frac{V_i(\psi_k \lambda_j(k)h^k_{l(k),j})}{||V_i||(M)} -  \sum_k\sum_i \sum_j \alpha_i \frac{V_i(\psi_k \lambda_j(k)h(u^k_j))}{||V_i||(M)}\Big| < C/m,
\end{equation}
%correction1: adding (ref) 
and
\begin{equation} \nonumber
 \Big|  \sum_k\sum_i \sum_j \alpha_i \frac{V_i(\psi_k \lambda_j(k)h(u^k_j))}{||V_i||(M)}- \sum_i \alpha_i \frac{V_i(h)}{||V_i||(M)}\Big| < C/m,
\end{equation}
so we conclude 
\begin{equation} \nonumber
 \Big| \sum_i \alpha_i \frac{V_i(h)}{||V_i||(M)} - \frac{\int_{M}\Tr h}{(n+1){\rm vol}(M)}\Big| < C/m.
\end{equation}

%correction1: precised it is computed with \tilde{g}
In the paragraph that follows, all the integrals, traces and varifold values are computed with $\tilde{g}$. Each $V_i=V_{m,i}$, $i=1, \dots, J_m=J$, is of the form
$$
V_i = \sum_{q=1}^{R_{m,i}} \Sigma_{m, i,q},
$$
with $R_{m,i} \in \mathbb{N}$, $\Sigma_{m,i,q}$ a connected, closed, smooth, embedded, minimal hypersurface of $(M,\tilde{g})$. Choose integers $c_{m,i}, d_m\in \mathbb{N}$ such that $\alpha_i=\alpha_{m,i}$ satisfies
$$
|\frac{\alpha_{m,i}}{||V_{m,i}||(M)}-\frac{c_{m,i}}{d_m}|< \frac{1}{mJ_m||V_{m,i}||(M)}.
$$
In particular, $|1-\sum_{i=1}^{J_m}\frac{c_{m,i}||V_{m,i}||(M)}{d_m}|<1/m$ and 
\begin{equation} \nonumber
 \Big| \sum_i \frac{c_{m,i}}{d_m} V_{m,i}(h) - \frac{\int_{M}\Tr h}{(n+1){\rm vol}(M)}\Big| < C/m.
\end{equation}
Hence
\begin{equation}\label{for.symmetric}
\lim_{m\rightarrow \infty} \frac{\sum_{i=1}^{J_m} c_{m,i}V_{m,i}(h)}{\sum_{i=1}^{J_m}c_{m,i}||V_{m,i}||(M)}  =\frac{\int_{M}\Tr h}{(n+1){\rm vol}(M)}
\end{equation}
for any symmetric $(0,2)$-tensor $h$. If we choose $h=f\cdot \tilde{g}$, with $f \in C^\infty(M)$, we get
\begin{equation} \label{for.function}
\lim_{m\rightarrow \infty} \frac{\sum_{i=1}^{J_m}\sum_{q=1}^{R_{m,i}}  c_{m,i}\int_{\Sigma_{m, i,q}} f}{\sum_{i=1}^{J_m} \sum_{q=1}^{R_{m,i}}c_{m,i}{\rm vol}(\Sigma_{m,i,q})}  =\frac{\int_{M}f}{{\rm vol}(M)}.
\end{equation}
Because 
$$V_{m,i}(h)=\sum_{q=1}^{R_{m,i}}\int_{\Sigma_{m,i,q}}h(\nu,\nu) = \sum_{q=1}^{R_{m,i}}\int_{\Sigma_{m,i,q}} (\Tr h- \Tr_{\Sigma_{m,i,q}}h),$$
we can combine (\ref{for.symmetric}) with (\ref{for.function}) to conclude
$$
\lim_{m\rightarrow \infty} \frac{\sum_{i=1}^{J_m}\sum_{q=1}^{R_{m,i}}  c_{m,i}\int_{\Sigma_{m, i,q}} \Tr_{\Sigma_{m,i,q}}h}{\sum_{i=1}^{J_m} \sum_{q=1}^{R_{m,i}}c_{m,i}{\rm vol}(\Sigma_{m,i,q})}  =\frac{n\int_{M}\Tr h}{(n+1){\rm vol}(M)}.
$$

%correction2: added following paragraph
In other words, we just proved that, assuming Property ({\bf P}), one can find for a generic metric a sequence of finite lists of closed, embedded, connected minimal hypersurfaces $\big\{\Sigma_{N,1}, \dots, \Sigma_{N, P_N}\big\}_{N\in\mathbb{N}}$ such that the following is true: if we denote $\int_{\Sigma_{N,i}} \Tr_{\Sigma_{N,i}} (h) \, d\Sigma_{N,i}$ (resp. $ {\rm vol}(\Sigma_{N,i}) $) by $X_{N,i}$ (resp. $\bar{X}_{N,i}$), then
\begin{equation} \label{group conv}
|\frac{\sum_{i=1}^{P_N} X_{N,i}}{\sum_{i=1}^{P_N} \bar{X}_{N,i}}-\alpha|\leq \varepsilon_N,%=\alpha:=\frac{1}{{\rm vol}(M)}\int_M \frac{n\Tr_M h}{n+1} dM,
\end{equation}
where $\alpha=\frac{1}{{\rm vol}(M)}\int_M \frac{n\Tr_M h}{n+1} dM$ and $\lim_{N \rightarrow \infty}\varepsilon_N=0$.
From the numbers $X_{N,i}$, $\bar{X}_{N,i}$, we want to construct two sequences $\{Y_j\}_{j\in\mathbb{N}}$ and $\{\bar{Y}_j\}_{j\in\mathbb{N}}$ %independent of $h$ 
such that 
\begin{itemize}
\item
for all $j$, there exist integers $N(j)$, $i(j)$ (chosen independently of h) with $Y_j = X_{N(j),i(j)}$ and $\bar{Y}_j = \bar{X}_{N(j),i(j)}$,
\item moreover $$\lim_{q \rightarrow \infty} \frac{\sum_{j=1}^{q} Y_j}{\sum_{j=1}^{q} \bar{Y}_{j}}=\alpha.$$
\end{itemize}

Note first that all the $\bar{X}_{N,i}$ are bounded below by a uniform positive constant $v$, according to the monotonicity formula, and that $|X_{N,i}|\leq C(h) \bar{X}_{N,i}$ where $C(h)$ is the maximum value that the absolute value of the trace of $h$ can take over the Grassmannian $G_n(M)$. 

Let $\{Q_N\}_{N\in \mathbb{N}}$ be a sequence of positive integers that will be chosen in the following order: $Q_1$ is chosen depending on 
$\{\Sigma_{1,i}\}$ and $\{\Sigma_{2,i}\}$, $Q_2$ is chosen depending on $Q_1$, $\{\Sigma_{1,i}\}$, $\{\Sigma_{2,i}\}$, $\{\Sigma_{3,i}\}$, and similarly $Q_{N_0}$ is chosen depending on $\{Q_1, \dots, Q_{N_0-1}\}$, $\{\Sigma_{1,i}\}$, $\{\Sigma_{2,i}\}$, $\dots \{\Sigma_{N_0+1,i}\}$.

  If $1\leq j \leq Q_1P_1$, write
$j=kP_1+l$ where $k\in \{0, \dots, Q_1-1\}$ and $l\in \{1,\dots,P_1\}$. Then define $Y_j=X_{1,l}$ and $\bar{Y}_j=\bar{X}_{1,l}$ accordingly. Notice that
\begin{eqnarray*}
&&|\frac{\sum_{j=1}^{kP_1+l}Y_j}{\sum_{j=1}^{kP_1+l}\overline Y_j}-\alpha|\\
&&\leq |\frac{k(\sum_{i=1}^{P_1} X_{1,i}-\alpha \sum_{i=1}^{P_1} \overline X_{1,i}) + (\sum_{i=1}^l X_{1,i}-\alpha\sum_{i=1}^l \overline X_{1,i})}{ k\sum_{i=1}^{P_1} \overline X_{1,i}+ \sum_{i=1}^l \overline X_{1,i}}|\\
&&\leq \varepsilon_1+C(h)+|\alpha|,
\end{eqnarray*}
while
\begin{eqnarray*}
&&|\frac{\sum_{j=1}^{Q_1P_1}Y_j}{\sum_{j=1}^{Q_1P_1}\overline Y_j}-\alpha|\leq \varepsilon_1.\\
\end{eqnarray*}

If $Q_1P_1+1 \leq j \leq Q_1P_1+Q_2P_2$, we write $j=Q_1P_1+kP_2+l$ where $k\in \{0, \dots, Q_2-1\}$ and $l\in \{1,\dots,P_2\}$. Then define $Y_j=X_{2,l}$  and $\bar{Y}_j=\bar{X}_{2,l}$ accordingly. Now
\begin{eqnarray*}
&&|\frac{\sum_{j=1}^{Q_1P_1+kP_2+l}Y_j}{\sum_{j=1}^{Q_1P_1+kP_2+l}\overline Y_j}-\alpha|\\
&&\leq \big| \big(Q_1(\sum_{i=1}^{P_1} X_{1,i}-\alpha \sum_{i=1}^{P_1} \overline X_{1,i}) + k(\sum_{i=1}^{P_2} X_{2,i}-\alpha \sum_{i=1}^{P_2} \overline X_{2,i})\\
&&\hspace{3cm} + (\sum_{i=1}^l X_{2,i}-\alpha\sum_{i=1}^l \overline X_{2,i})\big)\big|\\
&&\hspace{1cm}\cdot \frac{1}{\big(Q_1\sum_{i=1}^{P_1} \overline X_{1,i}+k\sum_{i=1}^{P_2} \overline X_{2,i}+ \sum_{i=1}^l \overline X_{2,i}\big)}\\
&&\leq \varepsilon_1+\varepsilon_2+\frac{C(h)+|\alpha|}{Q_1P_1v} \sum_{i=1}^{P_2} \overline X_{2,i}\\
&&\leq \varepsilon_1+\varepsilon_2+ (C(h)+|\alpha|)\varepsilon_2,
\end{eqnarray*}
if $Q_1$ is sufficiently large depending on $\{\Sigma_{1,i}\}$ and $\{\Sigma_{2,i}\}$,
while
\begin{eqnarray*}
&&|\frac{\sum_{j=1}^{Q_1P_1+Q_2P_2}Y_j}{\sum_{j=1}^{Q_1P_1+Q_2P_2}\overline Y_j}-\alpha|\leq 2\varepsilon_2,\\
\end{eqnarray*}
if $Q_2$ is sufficiently large depending on $Q_1$, $\{\Sigma_{1,i}\}$ and $\{\Sigma_{2,i}\}$.

Proceeding this way we get a sequence $\{Q_N\}$ and a sequence $\{Y_j\}$ defined so that if $1+ \sum_{N=1}^{N=N_0} Q_NP_N\leq j \leq \sum_{N=1}^{N=N_0+1} Q_NP_N$, we write $j=\sum_{N=1}^{N=N_0} Q_NP_N+kP_{N_0+1}+l$, where $k \in \{0, \dots, Q_{N_0+1}-1\}$ and 
$l \in \{1, \dots, P_{N_0+1}\}$, and set $Y_j=\Sigma_{N_0+1,l}$, $\bar{Y}_j=\bar{X}_{N_0+1,l}$. We will have
\begin{eqnarray*}
&&|\frac{\sum_{j=1}^{\sum_{N=1}^{N=N_0} Q_NP_N+kP_{N_0+1}+l}Y_j}{\sum_{j=1}^{\sum_{N=1}^{N=N_0} Q_NP_N+kP_{N_0+1}+l}\overline Y_j}-\alpha|\leq 2\varepsilon_{N_0}+\varepsilon_{N_0+1}+(C(h)+|\alpha|)\varepsilon_{N_0+1},
\end{eqnarray*}
and
\begin{eqnarray*}
&&|\frac{\sum_{j=1}^{\sum_{N=1}^{N=N_0+1} Q_NP_N}Y_j}{\sum_{j=1}^{\sum_{N=1}^{N=N_0+1} Q_NP_N}\overline Y_j}-\alpha|\leq 2\varepsilon_{N_0+1}.
\end{eqnarray*}

This implies
 $$\lim_{q \rightarrow \infty} \frac{\sum_{j=1}^{q} Y_j}{\sum_{j=1}^{q} \overline{Y}_{j}}=\alpha$$
 for any $h$, and we are done.

\medskip

{\it Proof of the Property} ({\bf P)}: Let $g$ be a smooth Riemannian metric, $\epsilon_1>0$ and $\tilde{K}>0$ be constants, and choose 
$$S=(K, \{\hat{B}_k\}, \{B_k\}, \{\phi_k\}, \{q_k\}, \{x^k_l\}) \in \mathcal{C}_{g, \tilde{K}, \epsilon_1}.$$  

Let $\mathcal{U}$ be a $C^\infty$ neighborhood of $g$.
Let $N=KL(n+1)$. Choose $\epsilon'>0$ sufficiently small and $q\geq N+3$ sufficiently large so that if $g'\in \Gamma_\infty$ satisfies $||g-g'||_{C^q}<\epsilon'$, then $g'\in \mathcal{U}$. For each $k$, $l$, $j$ we associate a variable $t_{k,l,j}\in[0,1]$ and we order them by lexicographical order on the indices. We can find a smooth $(0,2)$-tensor $\bar{h}^k_{l,j}$ so that $||\bar{h}^k_{l,j}-h^k_{l,j}||_{C^q}<\epsilon'$ and such that
$\{\bar{h}^k_{l,j}\}_{l,j}$ is linearly independent in a neighborhood of $q_k$ where $\phi_k$ is equal to 1 and $\phi_{k'}$ is zero for $k'\neq k$.

Consider the following $N$-parameter family of metrics. For a $t=(t_{k,l,j})\in [0,1]^N$, we define 
$$\hat{g}(t)= g+2\sum_{k,l,j} \psi_k t_{k,l,j} \bar{h}^k_{l,j}.$$ 
%If $\delta$ is sufficiently small, $\hat{g}:[0,\delta]^N \rightarrow \Gamma_\infty$ is an embedding.
As $t$ goes to zero, we have the following expansion
\begin{eqnarray} \label{taylor exp}
%correction1: added a name to this equation and changed its form
&&  {\rm vol}(M,\hat{g}(t))^{\frac{n}{n+1}} =    {\rm vol}_g(M)^{\frac{n}{n+1}} 
 \\
  &&+ \frac{n}{(n+1)}
  {{\rm vol}_g(M)^{-\frac{1}{n+1}}}\sum_{k,l,j} t_{k,l,j}\int_M \psi_k dv_g
+ o(||t||_1) + O(\epsilon'||t||_1),\nonumber
\end{eqnarray}
where $||t||_1=\sum_{k,l,j} |t_{k,l,j}|$.
Also
$$
 \frac{\partial}{\partial t_{k,l,j}} {\rm vol}(M,\hat{g}(t)) = \int_M \psi_k \Tr_{\hat{g}(t)}(\bar{h}^k_{l,j}) dv_{\hat{g}(t)}
 %correction1: no need of o(1)
 %+ o(1)
 =\int_M \psi_k dv_g + o(1)+O(\epsilon').$$
 
 We will say that a function $f:[0,\delta]^N\to \mathbb{R}$ is $\epsilon'$-close to another function $g:[0,\delta]^N\to \mathbb{R}$ if, when appropriately rescaled to be functions defined on $[0,1]^N$, they are at distance less than $\epsilon'$ in the $L^\infty$ norm, i.e. $$||\frac{1}{\delta}f_\delta-\frac{1}{\delta}g_\delta||_\infty<\epsilon'$$
with $f_\delta(s)= f(\delta s)$ and $g_\delta(s)=g(\delta s)$. By (\ref{taylor exp}), the function
$$f_0(t) : = \frac{{\rm vol}(M,\hat{g}(t))^{n/(n+1)}}{{\rm vol}_g(M)^{n/(n+1)}} - \frac{n}{(n+1)}\frac{1}{{\rm vol}_g(M)}\sum_{k,l,j} t_{k,l,j}\int_M \psi_k dv_g$$
is $C\epsilon'$-close to the constant function equal to $1$ on $[0,\delta]^N$, where $C=C(g)$ depends only on $g$ and might differ from line to line, if $\delta$ is sufficiently small.

 If $\delta>0$ is sufficiently small, we also have that $\hat{g}:[0,\delta]^N \rightarrow \Gamma_q$ is an embedding and $||\hat{g}(t)-g||_{C^q}<\epsilon'/2$ for every $t\in [0,\delta]^N$. We can slightly perturb $\hat{g}$ in the $C^\infty$ topology  into a $C^\infty$ map $g':[0,\delta]^N \rightarrow \Gamma_q$ so that the conclusion of Lemma \ref{differential} is satisfied. In particular, we can assume $||g'(t)-\hat{g}(t)||_{C^q}<\epsilon'/4$ and $||\frac{\partial g'}{\partial v}(t)-\frac{\partial \hat{g}}{\partial v}(t)||_{C^q} < \epsilon'/4$ for any $t \in [0,\delta]^N$ and $v\in \mathbb{R}^N$, $|v|=1$, and the function
 $$f_1(t) : = \frac{{\rm vol}(M,g'(t))^{n/(n+1)}}{{\rm vol}_g(M)^{n/(n+1)}} - \frac{n}{(n+1)}\frac{1}{{\rm vol}_g(M)}\sum_{k,l,j} t_{k,l,j}\int_M \psi_k dv_g$$
is $C\epsilon'$-close to the constant function equal to $1$ on $[0,\delta]^N$.

The normalized widths $p^{-\frac{1}{n+1}}\omega_p(g'(t))$ of $g'(t)$ ($t\in[0,\delta]^N$) are uniformly Lipschitz continuous on $[0,\delta]^N$ by Lemma \ref{lipschitz}. Hence, by the Weyl Law for the Volume Spectrum (\cite{liokumovich-marques-neves}), the functions $t \mapsto p^{-\frac{1}{n+1}}\omega_p(g'(t))$ converge uniformly to the function $t \mapsto a(n)\Vol(M,g'(t))^{\frac{n}{n+1}}$.  Hence if 
 $p$ is sufficiently large, $|p^{-\frac{1}{n+1}}\omega_p(g'(t)) - a(n)\Vol(M,g'(t))^{\frac{n}{n+1}}|<\delta\epsilon'$ and the function 
$$f_2(t) : = \frac{\omega_p(g'(t))}{a(n){\rm vol}_g(M)^{n/(n+1)}p^{1/(n+1)}} - \frac{n}{(n+1)}\frac{1}{{\rm vol}_g(M)}\sum_{k,l,j} t_{k,l,j}\int_M \psi_k dv_g$$
is $C\epsilon'$-close to the constant function equal to $1$ on $[0,\delta]^N$.

%correction1: C^q
   
%correction3: \hat{g} not embedding!
%Indeed although the $h^k_{l,j}$ are not linearly independent and $\hat{g}$ is not an embedding, one can first perturb each $h^k_{l,j}$ to $
%\bar{h}^k_{l,j}$ to make the $\psi_k \bar{h}^k_{l,j}$ linearly independent and then apply Lemma \ref{differential} to the modified family
%$$\bar{g}(t)= g+2\sum_{k,l,j} \psi_k t_{k,l,j} \bar{h} ^k_{l,j}$$
%which is an embedding if $\delta$ is small enough, by the Implicit Function Theorem. 

Then at each $t\in \mathcal{A}$ (where $\mathcal{A}$ is given by Lemma \ref{differential}), there is a varifold $V\in\mathcal{V}(g'(t))$ with support a minimal hypersurface $\Sigma$ such that 
\begin{align} \label{koala}
\begin{split}
\frac{\partial}{\partial t_{k,l,j}} p^{-\frac{1}{n+1}}\omega_p({g'}(t))=& p^{-\frac{1}{n+1}} \frac{\partial}{\partial t_{k,l,j}} ||V||(M,g'(t)) \\
=&p^{-\frac{1}{n+1}}||V||(\psi_k \Tr_\Sigma({\bar{h}^k_{l,j}}))+O(\epsilon')
%correction1: removed c(g'(t))
\\
=&p^{-\frac{1}{n+1}}\big(||V||(\psi_k\Tr_{M,g'(t)}\bar{h}^k_{l,j} ) -V(\psi_k{\bar{h}^k_{l,j}}) \big)+ O(\epsilon')\\
%correction1: added the parameter \delta
%correction3: put \{g'(t)\} instead of g'(t)
=&p^{-\frac{1}{n+1}}\big(||V||(\psi_k) -V(\psi_k{h^k_{l,j}}) \big)+O(\epsilon'),
\end{split}
\end{align}
where $||V||(.)$, $V(.)$ and the traces are computed with respect to $g'(t)$. 
%Here $c(g')$ goes to zero 
%uniformly in $t$ 
%as $\{g'(t)\}$ converges to the family $\{\hat{g}(t)\}$. 
%Note that in the previous formulae, $o(1)$ and $o(||t||_1)$ depend only on $\{\hat{g}(t)\}$ and are independent of $p$ and $\{g(t)\}$. 

Given $\eta>0$, we can choose $0<\epsilon'<\eta$ sufficiently small compared to $C=C(g)$
%correction1
%(and consequently $\delta$, $1/p$ sufficiently small and $\{g'(t)\}$ close enough to $\{\hat{g}(t)\}$) 
so that we can apply Lemma \ref{maximum point}  to $f_2$ and to $\mathcal{A}$. We  get  sequences of points  $\{y_{1,m}\}_m, \dots, \{y_{N+1,m}\}$ in $\mathcal{A}$ converging to a common limit $y \in (0,\delta)^N$ such that
%\begin{itemize}
%\item $f_0$ is differentiable at each $y_{i,m}$,
%\item 
the gradients $\nabla f_2(y_{i,m})$ converge to $N+1$ vectors $v_1,\dots,v_{N+1}$ with
$$
d_{\mathbb{R}^N}\left(0, \Conv(v_1,\dots,v_{N+1})\right) < \eta.
$$
%\end{itemize}

Let $\{\alpha_1, \dots, \alpha_{N+1}\} \subset [0,1]$ with $\sum_{i=1}^{N+1} \alpha_i=1$ such that $|\alpha_1v_1 + \cdots + \alpha_{N+1}v_{N+1}|< \eta$. Then for sufficiently large $m$, we have
$$
|\alpha_1\nabla f_2(y_{1,m}) + \cdots + \alpha_{N+1} \nabla f_2(y_{N+1,m})|< \eta,
$$
and hence
\begin{eqnarray}\label{convex.distance}
|\alpha_1\frac{\partial f_2}{\partial t_{k,l,j}}(y_{1,m}) + \cdots + \alpha_{N+1}\frac{\partial f_2}{\partial t_{k,l,j}}(y_{N+1,m})|< \eta
\end{eqnarray}
for all $k,l,j$.

According to Lemma \ref{differential}, each gradient based at $y_{i,m}$ corresponds to a varifold of mass $\omega_p(g'(y_{i,m}))$ whose support is a minimal hypersurface in $(M,g'(y_{i,m}))$ of index bounded by $p$. Hence for all $i$, by Sharp's Compactness Theorem (\cite{sharp}) a subsequence in $m$ of these varifolds converges to a varifold of $\mathcal{V}(g'(y))$ whose mass is $\omega_p(g'(y))$. By 
%correction1: added (koala)
(\ref{koala}) and (\ref{convex.distance}),  we have $N+1$ varifolds $V_i$ in $\mathcal{V}(g'(y))$ such that
$$\quad \Big|\sum_i \alpha_i  
%correction1: notation ||V_i||
\frac{||V_i||(\psi_k)-V_i(\psi_kh^k_{l,j})}{a(n){\rm vol}_g(M)^{n/(n+1)}p^{1/(n+1)}} - \frac{n}{(n+1)}\frac{1}{{\rm vol}_g(M)}\int_M \psi_k dv_g\Big| < C\eta$$
for all $k,l,j$, where $||V_i||(.)$ and $V_i(.)$ are computed with respect to $g'(y)$. This implies
$$\forall k,l,j \quad \Big|\sum_i \alpha_i  
%correction1: notation ||V_i||
\frac{||V_i||(\psi_k)-V_i(\psi_kh^k_{l,j})}{||V_i||(M)} - \frac{n}{(n+1)}\frac{1}{{\rm vol}_g(M)}\int_M \psi_k dv_g\Big| < C\eta.$$ We also have that   for all $i$, $k$, $l$, one has
$$\Big| \sum_{j=1}^{n+1} V_i(\psi_kh^k_{l,j})-||V_i||(\psi_k)\Big| = \Big|V_i(\psi_k g) - ||V_i||(\psi_k)\Big|< C\eta||V_i||(M),
%correction1: V_i is computed with g'(y) not g
$$ 
%correction1: 
$$\text{ and } \Big| \frac{1}{(n+1)}\frac{1}{{\rm vol}_g(M)}\int_M \psi_k dv_g - \frac{1}{(n+1)}\frac{1}{{\rm vol}_{g'(y)}(M)}\int_M \psi_k dv_{g'(y)} \Big|<C\eta.$$

We deduce the following:
\begin{equation} \label{10eta}
\forall k,l,j \quad \Big|\sum_i \alpha_i  \frac{V_i(\psi_kh^k_{l,j})}{||V_i||(M)} - 
%correction1: g'(y) instead of g
\frac{1}{(n+1)}\frac{1}{{\rm vol}_{g'(y)}(M)}\int_M \psi_k dv_{g'(y)}\Big| < 
%correction1: I get 10\eta as a safe bound
C\eta.
\end{equation}

The metric $g'(y)\in \Gamma_q$ satisfies $||g'(y)-g||_{C^q}<3\epsilon'/4$. We apply Lemma \ref{make hypersurface nondegenerate again} to $\bigcup_{i=1}^{N+1}\spt(V_i)$ and find a $C^q$ metric $\overline{g}$ such that $||\overline{g}-g||_{C^q}<4\epsilon'/5$,  each $\spt(V_i)$ is nondegenerate minimal with respect to $\overline{g}$ and (\ref{10eta}) is still valid with $g'(y)$ replaced by $\overline{g}$. If $\tilde{g}$ is a $C^\infty$ metric that is sufficiently close to $\overline{g}$ in the $C^q$ topology, then $\tilde{g} \in \mathcal{U}$. Because of the nondegeneracy of $\spt{V_i}$ with respect to $\overline{g}$ and the Implicit Function Theorem, we can also assure that there are varifolds $V_1,\dots,V_J$ of $\mathcal{V}(\tilde{g})$ 
%correction2: added nondegeneracy here, will be proved at the end 
whose support $\spt(V_j)$ are nondegenerate, and coefficients $\alpha_1,\dots,\alpha_J\in[0,1]$ with $\sum_i \alpha_i =1$ satisfying
\begin{equation} \label{10eta.2}
\forall k,l,j \quad \Big| \sum_i \alpha_i \frac{V_i(\psi_k h^k_{l,j})}{||V_i||(M)} - 
%correction1: precised it is computed with \tilde{g} instead of g, the reason for this choice is that the V() are always computed with \tilde{g}, so it seems less confusing to also compute the integrals with \tilde{g}
\frac{1}{(n+1)}\frac{1}{{\rm vol}_{\tilde{g}}(M)}\int_M \psi_k dv_{\tilde{g}}\Big| < C\eta,
\end{equation}
where the terms of the sum are computed for the metric $\tilde{g}$. Since $\eta$ is arbitrarily small, we have proved  Property ({\bf P}).

%correction3:make the hypersurfaces nondegenerate!
%Finally by choosing $\eta < \epsilon_1/(10K)$, we get Property ({\bf P}).

%Note that in particular if $\delta$ is small, for all $t$ the denominator $a(n)\Vol(M,g)^{n/(n+1)}p^{1/(n+1)}$ is close to $\omega_p(g(t))$ (this is %in fact a direct consequence of the Weyl Law and the continuity of $\omega_p$).

\bibliographystyle{amsbook}

\end{document}